\documentclass[reqno, 12pt]{article}

\usepackage{enumerate}
\usepackage{needspace}
\usepackage{latexsym}
\usepackage{amsmath}
\usepackage{amsfonts}
\usepackage{amsthm}
\usepackage{amssymb,mathtools}
\usepackage{newlfont}
\usepackage{graphics}
\usepackage{color}
\usepackage{float}
\usepackage{diagbox}
\usepackage{tocloft}
\usepackage{titlesec}
\usepackage{booktabs,array,longtable,enumitem,microtype}
\usepackage{extpfeil}
\usepackage{centernot}

\usepackage[linesnumbered,ruled,vlined]{algorithm2e}
\usepackage{url}

\usepackage[margin=1in]{geometry}

\usepackage[T1]{fontenc}
\usepackage{lmodern}
\usepackage{microtype}
\usepackage{rotating}
\usepackage{multirow}
\usepackage{extarrows}
\usepackage[sort,compress,numbers]{natbib}
\usepackage[utf8]{inputenc}
\usepackage{xcolor}
\usepackage{listings}
\usepackage{aliascnt}
\usepackage[pagebackref,colorlinks=true,linkcolor=blue,citecolor=red,urlcolor=blue]{hyperref}
\usepackage[nameinlink,noabbrev,capitalize]{cleveref}
\numberwithin{equation}{section}

\setlist[enumerate]{label=\textup{(\roman*)},leftmargin=*}
\newtheorem{theorem}{Theorem}[section]
\newaliascnt{proposition}{theorem}
\newtheorem{proposition}[proposition]{Proposition}
\aliascntresetthe{proposition}
\newaliascnt{lemma}{theorem}
\newtheorem{lemma}[lemma]{Lemma}
\aliascntresetthe{lemma}
\newaliascnt{corollary}{theorem}

\aliascntresetthe{corollary}

\theoremstyle{definition}
\newaliascnt{definition}{theorem}

\aliascntresetthe{definition}
\newaliascnt{remark}{theorem}

\aliascntresetthe{remark}
\newaliascnt{example}{theorem}
\newtheorem{example}[example]{Example}
\aliascntresetthe{example}
\newaliascnt{conjecture}{theorem}

\aliascntresetthe{conjecture}
\newaliascnt{question}{theorem}
\newtheorem{question}[question]{Question}
\aliascntresetthe{question}
\newaliascnt{problem}{theorem}

\aliascntresetthe{problem}
\newaliascnt{alg}{theorem}

\aliascntresetthe{alg}

\allowdisplaybreaks[4]

\SetKwInput{KwInput}{Input}                
\SetKwInput{KwOutput}{Output}              

\title{Counterexamples to the Mu--Welker recursive decomposition in every degree}

\author{Feihu Liu$^{\color{blue} \dag}$, Ying Wang$^{\color{blue} \ddag}$, and Zihao Zhang$^{\color{blue} \S}$
\\[2mm]
{\small $^{\color{blue} \dag}$ Center for Combinatorics, LPMC,}\\[-0.8ex]
{\small Nankai University, Tianjin 300071, P.R.~China}\\
{\small $^{\color{blue} \ddag}$ School of Mathematics and Statistics,}\\[-0.8ex]
{\small North China University of Water Resources and Electric Power, Zhengzhou, 450045, P.R.~China}\\
{\small $^{\color{blue} \S}$ School of Mathematics and Statistics,}\\[-0.8ex]
{\small Beijing Institute of Technology, Beijing 102400, P.R.~China}\\
{\small {\color{blue} $^\dag$} Email address: liufeihu7476@163.com}\\
{\small {\color{blue} $^\ddag$} Email address: wangying2019@ncwu.edu.cn}\\
{\small {\color{blue} $^\S$} Email address: zihao-zhang@foxmail.com}\\
}

\date{\today}

\begin{document}

\maketitle

\begin{abstract}
The well-known open problem of Bell and Skandera asks whether a real-rooted polynomial $f(t)$ with positive integer coefficients and constant term one is the $f$-polynomial of a simplicial complex. Mu and Welker proved that if the recursive decomposition $f(t)=g(t)+th(t)$ satisfies the corresponding coefficient inequality $h_i<g_i$, then this open problem has an affirmative answer. Mu and Welker also conjectured that the real-rootedness of $f(t)$ implies that of $g(t)$ and $h(t)$. We give counterexamples to the conjecture of Mu and Welker for every degree at least three, and prove that the assertion holds in degrees $1$ and $2$. Moreover, each polynomial we construct is the $f$-polynomial of a simplicial complex.
\end{abstract}

\noindent
\begin{small}
\end{small}

\noindent
\begin{small}
\end{small}


\section{Introduction}

A polynomial with real coefficients is \emph{real-rooted} if all its zeros are real, with multiplicities allowed. A nonzero constant polynomial is regarded as real-rooted.

A finite \emph{simplicial complex} is a nonempty collection $\Delta$ of subsets of a finite set that is closed under taking subsets. Its elements are called \emph{faces}, and the empty set is a face. A face with $i$ vertices has
dimension $i-1$. If $f_{i-1}^{\Delta}$ denotes the number of faces with $i$ vertices, then the \emph{$f$-polynomial} of $\Delta$ is
\[f^{\Delta}(t)=\sum_{\sigma\in\Delta}t^{|\sigma|}=1+\sum_{i\ge1}f_{i-1}^{\Delta}t^i.
\]
We recommend \cite{Stanley1996} and \cite[Chapter 12]{RP.StanleyAC} for background on simplicial complexes.
For more on characterizations of $f$-polynomial of simplicial complexes, see \cite{Katona,Kruskal,Macaulay}.

Bell and Skandera \cite[Question~1.1]{BS} asked whether every real-rooted polynomial with positive integer coefficients and constant coefficient $1$ is the $f$-polynomial of a simplicial complex. Mu and Welker \cite[Section~3.2]{MW} approached this realization problem through a recursive decomposition defined by the standard integer binomial expansions of the coefficients. They also asked whether this decomposition preserves real-rootedness. We give a negative answer in every degree at least three, using polynomials that themselves have a simple simplicial realization.

For simplicial complexes, the corresponding numerical characterization is the Kruskal--Katona theorem \cite{Katona,Kruskal}. We recall the \emph{binomial expansions} that occur in this theorem and in the \emph{recursive decomposition} of Mu and Welker. For nonnegative integers $a,k$, we use the convention $\binom{a}{k}=0$ if $a<k$ and $\binom{a}{0}=1$. For a positive integer
$d$, let
\[f(t)=1+\sum_{i=1}^{d}c_it^i,\qquad c_i\in\mathbb Z_{>0}.
\]
Fix $i\in\{1,\ldots,d\}$. The coefficient $c_i$ has a unique standard $i$-th binomial expansion with integer upper indices
\cite[Section~1]{Katona}:
\[c_i=\binom{a_i}{i}+\binom{a_{i-1}}{i-1}+\cdots+\binom{a_j}{j},
 \qquad
a_i>a_{i-1}>\cdots>a_j\ge j\ge1,\qquad j\le i.
\]
The integers $j,a_j,\ldots,a_i$ depend on $c_i$. The first upper index is characterized by
\begin{equation}\label{eq:binomial-leading}
 \binom{a_i}i\le c_i<\binom{a_i+1}i.
\end{equation}
In this notation, the Kruskal--Katona theorem states that $f(t)$ is the $f$-polynomial of a simplicial complex if and only if
\[\sum_{r=j}^{i}\binom{a_r}{r-1} \leq c_{i-1}\qquad\text{for every}\qquad2\le i\le d,
\]
where the expansion on the left is taken separately for each $c_i$.

Mu and Welker \cite[Section~3.2]{MW} use these same expansions to define the coefficients of their decomposition.
Set
\begin{equation}\label{eq:recursive-coefficients}
 g_i=\sum_{r=j}^{i}\binom{a_r-1}{r}
 \quad\text{and}\quad
 h_{i-1}=\sum_{r=j}^{i}\binom{a_r-1}{r-1}.
\end{equation}
We define $g(t)$ and $h(t)$ by 
\[g(t)=1+\sum_{i=1}^{d}g_it^i,
 \qquad
 h(t)=\sum_{i=1}^{d}h_{i-1}t^{i-1}.
\]
For each term in the expansion of $c_i$, Pascal's identity gives
\[\binom{a_r-1}{r}+\binom{a_r-1}{r-1}=\binom{a_r}r.
\]
Summing over $r$ gives $c_i=g_i+h_{i-1}$ for each $i$, and hence
\[f(t)=g(t)+th(t),
\]
which is the \emph{recursive decomposition} of Mu and Welker. In particular, the two polynomials are uniquely determined by $f(t)$.
The expansion $c_1=\binom{c_1}{1}$ gives $g_1=c_1-1$ and $h_0=1$.
Every $g_i$ is nonnegative. Every term contributing to $h_{i-1}$ is positive, since $a_r-1\ge r-1$. Thus $g(t)$ has nonnegative coefficients and degree at most $d$, whereas $h(t)$ has positive coefficients and degree $d-1$; both have constant coefficient $1$.

For real-rooted $f(t)$, Mu and Welker conjectured that $h_i\le g_i$ for $1\le i\le d-1$, and proved that this conjecture would imply an affirmative answer to the Bell--Skandera question \cite[Conjecture~3.8 and Proposition~3.9]{MW}. Thus the decomposition
was introduced to obtain coefficient inequalities sufficient for simplicial realization. Motivated by calculations for Eulerian
polynomials and Stirling polynomials of the second kind, Mu and Welker posed the following open problem.

\begin{question}(Mu--Welker; \cite[Question~3.10]{MW})\label{Question-MU-Welker}
Let $f(t) = 1 + f_0t + \cdots + f_{d-1}t^d \in \mathbb{N}[t]$ be a polynomial with only real roots and let $f(t) = g(t) + th(t)$ its recursive decomposition. Is true that $g(t)$ and $h(t)$ both have only real roots?
\end{question}

The following theorem disproves this preservation property.

\begin{theorem}\label{thm:main}
Let $d\ge3$ and $m\ge12d$ be integers, and let
\[f_{d,m}(t)=(1+mt)^d=g_{d,m}(t)+t h_{d,m}(t)
\]
be the recursive decomposition defined by \eqref{eq:recursive-coefficients}, with $c_i=\binom{d}{i}m^i$ for
$1\le i\le d$. Then each of $g_{d,m}$ and $h_{d,m}$ has a pair of nonreal conjugate zeros.
\end{theorem}

For each $d\ge3$, the real-rooted polynomials $(1+mt)^d$, with integers
$m\ge12d$, therefore form an infinite family of counterexamples to \cref{Question-MU-Welker}.
\cref{prop:low-degrees} shows that the decomposition does preserve real-rootedness in degrees one and two. Hence degree three is the smallest possible degree of a counterexample.

The conditions in \cref{thm:main} already satisfy the conclusion sought by Bell and Skandera. 
Indeed, partition $dm$ vertices into $d$ sets of size $m$, and take as faces the subsets meeting each set in at most one vertex. The resulting complex has $f$-polynomial $(1+mt)^d$, as verified in \cref{Section-666}.
Thus the failure persists when the input is required to be an $f$-polynomial. 

Moreover, the identity
\[(1+mt)^d =(1+(m-1)t)(1+mt)^{d-1}+t(1+mt)^{d-1}\]
gives another decomposition of the same form in which both component polynomials are real-rooted. The obstruction therefore concerns the prescribed binomial rule. A recursive argument using this rule cannot assume real-rootedness of its successive components. The coefficient comparisons in \cite[Conjecture~3.8]{MW} are a separate issue and are not contradicted by our result.

The paper is organized as follows.
\cref{Section-Two} states two required inequalities.
In \cref{Section-Three}, we give the inequalities satisfied by the coefficients of the polynomials in \cref{thm:main}.
\cref{sec:h} is devoted to the proof, contained in \cref{thm:main}, that $h_{d,m}(t)$ is not real-rooted.
In \cref{Section=five}, we discuss the proof of the non-real-rootedness of $g_{d,m}(t)$.
In \cref{Section-666}, we proved that the polynomial $(1+mt)^d$ is the $f$-polynomial of a simplicial complex of dimension $d-1$.

\section{Two elementary inequalities}\label{Section-Two}

For the classical inequalities and their role in combinatorics, see \cite{HLP,Stanley}.
The first inequality below is the first nontrivial Newton inequality
for a polynomial with nonnegative coefficients; see, for example,
\cite{HLP}. Only this special case is needed, and its proof is included.

\begin{lemma}\label{lem:newton}
Let $P(t)=1+b_1t+b_2t^2+\cdots+b_nt^n$ have nonnegative real coefficients, with $n\ge2$ and $b_n>0$.
If $P$ is real-rooted, then
\[ b_1^2\ge\frac{2n}{n-1}b_2.
\]
\end{lemma}
\begin{proof}
For $t\ge0$ we have $P(t)\ge1$, so all real zeros of $P(t)$ are negative. By factoring $P(t)$ and using its constant coefficient, we may write
\[P(t)=\prod_{j=1}^{n}(1+\lambda_jt),
 \qquad \lambda_j>0.
\]
Comparison of coefficients gives
\[b_1=\sum_{j=1}^{n}\lambda_j,
 \qquad
 b_2=\sum_{1\le i<j\le n}\lambda_i\lambda_j.
\]
Thus
\[b_1^2=\sum_{j=1}^{n}\lambda_j^2+2b_2.
\]
On the other hand, we have 
\[n\sum_{j=1}^{n}\lambda_j^2-b_1^2=n\sum_{j=1}^{n}\lambda_j^2-\left(\sum_{j=1}^{n}\lambda_j\right)^2=\sum_{1\le i<j\le n}(\lambda_i-\lambda_j)^2\ge0.
\]
Hence $\sum_{j=1}^n\lambda_j^2\ge b_1^2/n$. Substituting this into the
preceding identity gives $b_1^2\ge b_1^2/n+2b_2$. Multiplication by
$n/(n-1)$ yields the assertion.
\end{proof}

\begin{lemma}\label{lem:moments}
For real numbers $z_1,\ldots,z_n$, one has
\[
 \left(\sum_{j=1}^{n}z_j^3\right)^2
 \le\left(\sum_{j=1}^{n}z_j^2\right)^3.
\]
\end{lemma}

\begin{proof}
Since $z_j^2\le\sum_{i=1}^n z_i^2$ for each $j$, i.e., $z_j\le(\sum_{i=1}^n z_i^2)^{\frac{1}{2}}$, we have
\[\left|\sum_{j=1}^{n}z_j^3\right|\le\sum_{j=1}^{n}|z_j|^3\le\left(\sum_{i=1}^{n}z_i^2\right)^{1/2}\sum_{j=1}^{n}z_j^2.
\]
Both sides are nonnegative, so squaring proves the result. The same argument also covers the case in which all $z_j$ vanish.
\end{proof}

\section{The quadratic and cubic coefficients}\label{Section-Three}

Fix integers $d\ge3$ and $m\ge1$. In this section and the next two, we omit the subscripts $d,m$ from $f_{d,m}(t)$, $g_{d,m}(t)$, and $h_{d,m}(t)$. Since $f(t)=g(t)+th(t)$ and $h_0=1$, the first three nonconstant coefficients
of $g(t)$ satisfy
\begin{equation}\label{eq:first-coefficients}
 \begin{gathered}
 g_1=dm-1,\qquad
 g_2=\binom{d}{2}m^2-h_1,\qquad
 g_3=\binom{d}{3}m^3-h_2.
 \end{gathered}
\end{equation}
The estimates needed below involve only $h_1$ and $h_2$. For real $x$, we use the polynomial notation $\binom{x}{2}=x(x-1)/2$.

\begin{lemma}\label{lem:coefficient-bounds}
For all integers $d\ge3$ and $m\ge1$, one has
\begin{equation}\label{eq:h1-bounds}
 (d-1)m\le h_1<m\sqrt{d(d-1)}+\frac12
\end{equation}
and
\begin{equation}\label{eq:h2-bound}
h_2>
\binom{\left(6\binom{d}{3}\right)^{1/3}m-1}{2}.
 \end{equation}
\end{lemma}

\begin{proof}
Recall that $c_2 = \binom{d}{2} m^2$. Let $a$ be the unique largest integer satisfying
\[
\binom{a}{2} \le \binom{d}{2} m^2 < \binom{a+1}{2},
\]
and define the remainder $b = \binom{d}{2} m^2 - \binom{a}{2}$.
Subtracting $\binom{a}{2}$ throughout gives
$
0 \le b < \binom{a+1}{2} - \binom{a}{2} = a.
$

In the standard binomial expansion $c_2 = \binom{a}{2} + b$, the  definition of the coefficients of $h(t)$ (see \eqref{eq:recursive-coefficients}) gives $h_1 = a-1$ if $b=0$, and $h_1 = (a-1)+1 = a$ if $b>0$. Thus $h_1 \ge a-1$ holds in both cases.

We first establish the lower bound $a \ge (d-1)m + 1$. Direct computation shows
\[
\binom{d}{2} m^2 - \binom{(d-1)m + 1}{2} = \frac{(d-1)m(m-1)}{2} \ge 0,
\]
so $\binom{(d-1)m + 1}{2} \le \binom{d}{2} m^2$. By the maximality of $a$, we conclude $a \ge (d-1)m + 1$. Combining this with the bound above gives
\[
h_1 \ge a - 1 \ge (d-1)m.
\]

For the upper bound, if $b>0$ then
\[
 h_1(h_1-1)=a(a-1)=d(d-1)m^2-2b
 \le d(d-1)m^2-2.
\]
If $b=0$, then $h_1=a-1$ and
\[
h_1(h_1-1) = (a-1)(a-2) \leq a(a-1)-2 = 
 d(d-1)m^2-2.
\]
The inequality holds because $a\ge d\ge3$.
Thus, in either case, we have 
\[\left(h_1-\frac12\right)^2=h_1(h_1-1)+\frac14\le d(d-1)m^2-\frac74 <d(d-1)m^2.
\]
Since $h_1\ge(d-1)m\ge2$, taking positive square roots gives the upper bound in \eqref{eq:h1-bounds}.

For the  coefficient $c_3$, let $A$ be the largest integer such that
\[
\binom{A}{3}\le \binom{d}{3}m^3<\binom{A+1}{3}<\frac{A^3}{6},
\]
and hence
\[
A>\left(6\binom{d}{3}\right)^{1/3}m.
\]
The leading term $\binom{A}{3}$ in the cubic binomial expansion contributes
$\binom{A-1}{2}$ to $h_2$, while all remaining contributions are
nonnegative. Therefore,
\[h_2\ge \binom{A-1}{2}.
\]
Moreover, $\left(6\binom{d}{3}\right)^{1/3} m\ge\sqrt[3]{6}>\frac{3}{2}$. By the strict monotonicity
of $x\mapsto\binom{x-1}{2}$ on $(\frac{3}{2},\infty)$, it follows that
\[h_2\ge\binom{A-1}{2}>\binom{\left(6\binom{d}{3}\right)^{1/3} m-1}{2},
\]
which proves \eqref{eq:h2-bound}.
\end{proof}

\begin{lemma}\label{lem:g-degree}
If $d\ge3$ and $m\ge2$, then $\deg g(t)=d$.
\end{lemma}
\begin{proof}
By the definition of $g(t)$, it suffices to show $g_d>0$. The inequality $2^d\ge d+1$ holds for $d=1$ and is preserved when $d$
is increased by one: $2^{d+1}\ge2(d+1)\ge d+2$. Hence $m^d\ge2^d\ge d+1=\binom{d+1}{d}$.
Let $a_d$ be the first upper index in the standard $d$-th binomial
expansion of $m^d$. Then $a_d\ge d+1$ by
\eqref{eq:binomial-leading}, and hence
\[g_d\ge\binom{a_d-1}{d}\ge\binom{d}{d}=1.
\]
This proves the degree assertion.
\end{proof}

\section{Failure of real-rootedness for $h(t)$}\label{sec:h}

\begin{proof}[Proof of the assertion for $h(t)$ in \cref{thm:main}]
Let $d\ge3$ and $m\ge12d$, and write $h(t)=h_{d,m}(t)$.
Since $h(t)$ has degree $d-1$, positive coefficients, and constant
term $1$, \cref{lem:newton} shows that it suffices to prove
\[ \Delta:=\frac{2(d-1)}{d-2}h_2-h_1^2>0.
\]
Set
\[
 s=\sqrt{d(d-1)},\qquad r=\Big(d(d-1)(d-2)\Big)^{1/3}.
\]
By \cref{lem:coefficient-bounds}, we have 
\[ h_1<sm+\frac12,
 \qquad h_2>\binom{rm-1}{2}.
\]
Since $h_1>0$, these estimates give
\[\Delta>\frac{2(d-1)}{d-2}\binom{rm-1}{2}-\left(sm+\frac12\right)^2=M_2m^2-M_1m+M_0,
\]
where
\[
 M_2=\frac{d-1}{d-2}r^2-s^2,\qquad
 M_1=\frac{3(d-1)}{d-2}r+s,\qquad
 M_0=\frac{2(d-1)}{d-2}-\frac14>0.
\]
Since $M_0>0$, it remains to show that $M_2m>M_1$.
We therefore bound $M_2$ from below and $M_1$ from above.

For $M_2$, set $X=\frac{d-1}{d-2}r^2$ and $Y=s^2$.
Then
\[
 X^3-Y^3=\frac{d^2(d-1)^3}{d-2}>0,
\]
so $X>Y>0$. Hence
\[ M_2=\frac{X^3-Y^3}{X^2+XY+Y^2}>\frac{X^3-Y^3}{3X^2}=\frac{d}{3r}>\frac{d}{3(d-1)},
\]
where the last inequality follows from $r^3=(d-1)^3-(d-1)<(d-1)^3$.

For $M_1$, observe that
\[\frac{d-1}{d-2}r=d\left(1+\frac{1}{d(d-2)}\right)^{2/3}\le d+\frac{2}{3(d-2)},
\]
where the inequality follows from the concavity bound
$(1+x)^{2/3}\le1+\frac23x$ for $x\ge0$.
Since 
\[ \left(d-\frac12\right)^2-s^2=\frac14>0,
\]
we have $s<d-\frac12$. By combining these estimates and using $d\ge3$, we obtain
\[M_1<4d+\frac{2}{d-2}-\frac12\le4d+\frac32.
\]

Finally, since $m\ge12d$, we have 
\[ M_2m>\frac{dm}{3(d-1)}\ge\frac{4d^2}{d-1}=4d+\frac{4d}{d-1}>4d+\frac32>M_1.
\]
Therefore, we have 
\[\Delta>m(M_2m-M_1)+M_0>0.
\]
This contradicts the necessary inequality in \cref{lem:newton} if $h(t)$ is real-rooted. Thus $h_{d,m}(t)$ has a nonreal zero.
\end{proof}

\section{Failure of real-rootedness for $g(t)$}\label{Section=five}

Recall the following two notations used in \cref{sec:h}:
\[s=\sqrt{d(d-1)},\qquad r=\Big(d(d-1)(d-2)\Big)^{1/3}.
\]

\begin{lemma}\label{lem:radical-gap}
For every integer $d\ge3$, we have 
\[(d-2)s-\binom{d-1}{2}-\frac{r^2}{2}<-\frac{d-1}{4d}.
\]
\end{lemma}
\begin{proof}
By the lower bound established in \cref{sec:h}, we obtain 
\[
 r^2-d(d-2)
 >\frac{d(d-2)}{3(d-1)^2}
 =\frac13\left(1-\frac1{(d-1)^2}\right)
 \ge\frac14.
\]
The identity $(d-1/2)^2-s^2=1/4$ also gives
\[
 d-\frac12-s
 =\frac{1}{4(d-\frac12+s)}
 >\frac1{8d}.
\]
Therefore, we have 
\[\begin{aligned}
 (d-2)s-\binom{d-1}{2}-\frac{r^2}{2}
 &=(d-2)\left(s-d+\frac12\right)
   -\frac12\bigl(r^2-d(d-2)\bigr)\\
 &<-\frac{d-2}{8d}-\frac18
 =-\frac{d-1}{4d}.
 \end{aligned}
\]
This completes the proof.
\end{proof}

\begin{proof}[Proof of the assertion for $g(t)$ in \cref{thm:main}]
Let $d\ge3$ and $m\ge12d$ be integers, and suppose, to the contrary, that $g(t)$ is real-rooted.
By \cref{lem:g-degree}, its degree is $d$. Its nonnegative coefficients and constant coefficient $1$ force all its zeros to be negative, so we may write
\[g(t)=\prod_{j=1}^{d}(1+(m+z_j)t),\qquad z_j\in\mathbb R,\quad z_j>-m.
\]
The shift by $m$ cancels the leading terms in the coefficient comparisons below.
Set
\[
 e_2=\sum_{1\le i<j\le d}z_iz_j \quad \text{and}\quad 
 e_3=\sum_{1\le i<j<k\le d}z_iz_jz_k.
\]
By comparing the coefficients of $t,t^2,t^3$ with \eqref{eq:first-coefficients}, we have
\begin{align*}
e_1=\sum_{j=1}^d z_j=-1,\quad e_2=(d-1)m-h_1\le0,\quad e_3=(d-2)mh_1-\binom{d-1}{2}m^2-h_2.
\end{align*}
The inequality for $e_2$ follows from \eqref{eq:h1-bounds}.

According to $(\sum_{j=1}^d z_j)^2=1$ and \eqref{eq:h1-bounds}, we obtain 
\[\sum_{j=1}^d z_j^2=1-2e_2=1+2h_1-2(d-1)m <m+2\le\frac98m,
\]
where we used $s=\sqrt{d(d-1)}<d-1/2$ and $m\ge12d\ge36$.

For the third power sum, we first estimate $e_3$ using \cref{lem:coefficient-bounds}:
\[
 \begin{aligned}
 e_3
 &<\left((d-2)s-\binom{d-1}{2}-\frac{r^2}{2}\right)m^2
   +\left(\frac{d-2}{2}+\frac32r\right)m-1\\
 &<-\frac{d-1}{4d}m^2+\left(2d-\frac52\right)m-1\\
 &\le-\frac{2d-1}{24d}m^2-1
 <-\frac5{72}m^2.
 \end{aligned}
\]
The second inequality uses \cref{lem:radical-gap} and $r<d-1$;
the third uses $m\ge12d$. The last follows from
$(2d-1)/(24d)\ge5/72$ for $d\ge3$.

According to $(\sum_{j=1}^d z_j)^3=-1$, we have 
\[\sum_{j=1}^d z_j^3=-1+3e_2+3e_3 <-\frac5{24}m^2,
\]
since $e_2\le0$. This power sum is negative, so squaring its bound and
using $m\ge36$ yields
\[\begin{aligned}
\left(\sum_{j=1}^d z_j^3\right)^2>\frac{25}{576}m^4\ge\frac{25}{16}m^3>\frac{729}{512}m^3>\left(\sum_{j=1}^d z_j^2\right)^3.
\end{aligned}
\]
This contradicts \cref{lem:moments}, proving the assertion for $g(t)$.
The assertion for $h(t)$ was proved in \cref{sec:h}; both polynomials have real coefficients, so their nonreal zeros occur in conjugate pairs.
\end{proof}

\section{The underlying simplicial complexes}\label{Section-666}

\begin{proposition}\label{prop:low-degrees}
Let $f(t)$ be a real-rooted polynomial of degree one or two with positive integer coefficients and constant coefficient $1$. Both polynomials in its recursive decomposition are real-rooted.
\end{proposition}
\begin{proof}
If $f(t)=1+c_1t$, its decomposition is
\[
 g(t)=1+(c_1-1)t,\qquad h(t)=1.
\]
The assertion follows, including the case $c_1=1$, in which $g(t)$ is constant.

Now let $f(t)=1+c_1t+Nt^2$, where $c_1,N$ are positive integers.
Real-rootedness is equivalent to $c_1^2\ge4N$. Write the standard
quadratic expansion as
\[
 N=\binom{a}{2}+b,\qquad 0\le b<a,
\]
and recall that $h_1$ is the coefficient of $t$ in $h(t)$. If $b=0$, then $h_1=a-1$ and $N=\binom{h_1+1}{2}$. If $b>0$, then $h_1=a$, and the integrality of $b<a=h_1$ gives
\[ N=\binom{h_1}{2}+b
 \le\binom{h_1}{2}+h_1-1=\binom{h_1+1}{2}-1.
\]
Thus in both cases
\[ N\le\binom{h_1+1}{2}=\frac{h_1(h_1+1)}{2}\le h_1^2,
\]
where $h_1\ge1$ justifies the final inequality. In particular,
$h_1\ge\sqrt N$.

The recursive decomposition is
\[ g(t)=1+(c_1-1)t+(N-h_1)t^2,\qquad h(t)=1+h_1t.
\]
The coefficient $N-h_1$ is nonnegative by the definition of the
decomposition. If it vanishes, $g(t)$ is linear because
$c_1\ge2\sqrt N\ge2$, and hence it is real-rooted.
If $N-h_1>0$, the discriminant of $g(t)$ satisfies
\begin{align*}
 (c_1-1)^2-4(N-h_1)\ge(2\sqrt N-1)^2-4(N-h_1)=1+4(h_1-\sqrt N)\ge1.
\end{align*}
The first inequality follows from
$c_1-1\ge2\sqrt N-1\ge1$ by squaring nonnegative quantities.
Thus $g(t)$ is real-rooted in either case. The polynomial $h(t)$ is linear,
so it is also real-rooted.
\end{proof}

\begin{example}
A small counterexample for \cref{Question-MU-Welker} is
\[f(t)=(1+3t)^3=1+9t+27t^2+27t^3.
\]
The relevant standard binomial expansions are
\[
 9=\binom91,\qquad
 27=\binom72+\binom61,\qquad
 27=\binom63+\binom42+\binom11.
\]
For the quadratic expansion, $\binom72=21\le27<28=\binom82$, and
the remainder is $6<7$. For the cubic expansion,
$\binom63=20\le27<35=\binom73$; the remainder is $7$, whose
quadratic expansion is $\binom42+\binom11$. Thus the displayed
expansions are indeed the prescribed ones.
Applying \eqref{eq:recursive-coefficients} gives
\[
 g(t)=1+8t+20t^2+13t^3=(1+t)(1+7t+13t^2),
 \qquad h(t)=1+7t+14t^2.
\]
The discriminants of the two quadratic polynomials are respectively
$49-52=-3$ and $49-56=-7$. Hence both output polynomials have nonreal
zeros. This example also shows that the sufficient bounds in \cref{thm:main} need not be close to the smallest admissible
parameter in a fixed degree.
\end{example}

\begin{proposition}
For every pair of positive integers $d,m$, the polynomial $(1+mt)^d$
is the $f$-polynomial of a simplicial complex of dimension $d-1$.
\end{proposition}

\begin{proof}
Let $V_1,\ldots,V_d$ be pairwise disjoint sets with $|V_j|=m$ for every
$j$, and let
\[
 \Delta=\left\{\sigma\subseteq V_1\cup\cdots\cup V_d:
                 |\sigma\cap V_j|\le1\text{ for }1\le j\le d\right\}.
\]
The empty set belongs to $\Delta$. Every subset of a member of $\Delta$
also meets each $V_j$ in at most one element, so $\Delta$ is a simplicial
complex. A face has at most $d$ vertices, and a face with $d$ vertices
exists by choosing one vertex from each nonempty set $V_j$. Thus
$\Delta$ has dimension $d-1$.

To choose a face with $i$ vertices, first choose the $i$ sets among
$V_1,\ldots,V_d$ that it meets, and then choose one vertex from each of
these sets. This gives exactly $\binom{d}{i}m^i$ faces. The count also
holds for $i=0$, where it gives the unique empty face. Hence
\[f^{\Delta}(t)=\sum_{i=0}^{d}\binom{d}{i}m^it^i=(1+mt)^d.
\]
This completes the proof.
\end{proof}

The coefficient inequalities in \cite[Conjecture~3.8]{MW} also hold for
this family, as follows from the colex construction in the
Kruskal--Katona theorem.

\begin{proposition}
For all positive integers $d,m$, the recursive decomposition
\[(1+mt)^d=g(t)+th(t)
\]
satisfies $h_i\le g_i$ for $0\le i\le d-1$.
\end{proposition}
\begin{proof}
For finite sets $A,B$, write $A\triangle B=(A\setminus B)\cup(B\setminus A)$ for their symmetric difference. The \emph{colexicographic order}, or \emph{colex order}, on subsets of the positive integers of a fixed
cardinality is defined, for $A\ne B$, by
\[A<_{\mathrm{colex}}B
 \quad\Longleftrightarrow\quad
 \max(A\triangle B)\in B.
\]

By the preceding proposition, the numbers
\[
 c_i=\binom di m^i,\qquad 0\le i\le d,
\]
are the face numbers of a simplicial complex. The colex form of the
Kruskal--Katona theorem \cite{Katona,Kruskal} therefore gives a
simplicial complex $\Gamma$ whose $i$-element faces are the first $c_i$
$i$-element subsets in colex order, for $1\le i\le d$.
Indeed, among all families of $c_i$ sets of cardinality $i$, the colex
initial segment minimizes the number of their $(i-1)$-element subsets,
and these subsets themselves form a colex initial segment.
Since the numbers $c_i$ are realized by a simplicial complex, this
number is at most $c_{i-1}$. Thus the chosen initial segments, together
with the empty set, are closed under taking subsets.

Fix $1\le i\le d$, and write the standard binomial expansion
\[
 c_i=\sum_{r=j}^{i}\binom{a_r}{r},
 \qquad a_i>a_{i-1}>\cdots>a_j\ge j\ge1.
\]
Writing $[a]=\{1,\ldots,a\}$, the corresponding colex initial segment
is the disjoint union of the blocks
\[\left\{\{a_i+1,\ldots,a_{r+1}+1\}\cup B:B\subseteq[a_r],\ |B|=r\right\},\qquad j\le r\le i,
\]
where the fixed set is empty when $r=i$. All vertices in the fixed
set exceed $a_r$, so that set does not contain $1$. The numbers of
faces in the $r$-th block avoiding and containing $1$ are therefore
$\binom{a_r-1}{r}$ and $\binom{a_r-1}{r-1}$, respectively. Summing over
the blocks and using \eqref{eq:recursive-coefficients} gives
\[
 \begin{aligned}
 g_i&=\#\{\sigma\in\Gamma:|\sigma|=i,\ 1\notin\sigma\},\\
 h_{i-1}&=\#\{\sigma\in\Gamma:|\sigma|=i,\ 1\in\sigma\}.
 \end{aligned}
\]
For $1\le i\le d-1$, deleting $1$ from an $(i+1)$-element face
containing $1$ gives an $i$-element face avoiding $1$, since $\Gamma$
is closed under taking subsets. This map is injective, proving
$h_i\le g_i$. Finally, $h_0=g_0=1$.
\end{proof}

Thus the examples disprove the real-rootedness assertion in \cite[Question~3.10]{MW} while satisfying the coefficient inequalities
in \cite[Conjecture~3.8]{MW}. Their input polynomials also satisfy the conclusion sought in \cite[Question~1.1]{BS}. In particular, the inequalities $h_i\le g_i$ do not ensure real-rootedness of either polynomial in the recursive decomposition.






\noindent
{\small \textbf{Acknowledgments:}}
The authors would like to express sincere gratitude for all the suggestions that have improved the presentation of this paper.
Feihu Liu was partially supported by the Postdoctoral Fellowship Program and China Postdoctoral Science Foundation (Grant No. BX2026002).
Ying Wang was partially supported by the Natural Science Foundation of Henan Province (Grant No. 262300422649).

\noindent{\small \textbf{Declaration of AI Assistance:}}
During the preparation of this manuscript, the authors used ChatGPT to assist in exploring possible approaches, checking technical details, and improving the exposition. All mathematical arguments, computations, proofs, and results were independently verified by the authors. The authors take full responsibility for the content of this manuscript.



\end{document}